\documentclass{article}
\usepackage{amsfonts}
\usepackage{amsmath}

\newtheorem{theorem}{Theorem}

\newtheorem{lemma}[theorem]{Lemma}

\newtheorem{proposition}[theorem]{Proposition}

\newenvironment{proof}[1][Proof]{\noindent\textbf{#1.} }{\ \rule{0.5em}{0.5em}}
\input{tcilatex}

\begin{document}

\title{Biharmonic maps from Finsler spaces}
\author{Nicoleta VOICU \\
"Transilvania"\ University\\
50, Iuliu Maniu str., Brasov, Romania\\
e-mail: nico.brinzei@unitbv.ro}
\date{}
\maketitle

\begin{abstract}
The notions of bienergy of a smooth mapping and of biharmonic map between
Riemannian manifolds are extended to the case when the domain is Finslerian.
We determine the first and the second variation of the bienergy functional,
the equations of Finsler-to-Riemann biharmonic maps and some specific
examples. We prove that two notable results in Riemannian geometry
concerning the inexistence of nonharmonic biharmonic maps still hold true
this case.
\end{abstract}



\section{Introduction}

Biharmonic mappings, as a generalization of harmonic ones, are among the
most important mappings in physics; initially appearing from problems of
elasticity theory and fluid mechanics, \cite{Selvadurai}, in the latter
decades, they proved to be useful also in computer graphics, geometry
processing, \cite{Lipman} and radar imaging, \cite{Andersson}. Mathematical
arguments, \cite{Montaldo}, for the use of biharmonic maps include the fact
that harmonic maps do not always exist - and biharmonic maps can "succeed
where harmonic maps have failed"- together with stability issues. On the
other side, Finslerian models seem to gain more and more ground wherever
anisotropy of some kind is involved (and not only), from domains such as:
kinematics, elasticity theory, \cite{Bucataru}, seismic ray theory, \cite%
{Anto-Slawinski}, \cite{Yajima}, \cite{Yajima1}, gravity theories, \cite{Li}%
, \cite{Munteanu}, \cite{Vacaru}, geometrical optics, \cite{Anto},
thermodynamics, statistical mechanics, \cite{Anto}, \cite{Ootsuka}, and up
to biology, \cite{Anto}, \cite{Anto-biology}.

While in Riemannian geometry, biharmonic mappings have been quite
intensively studied (see, for instance, \cite{Balmus}, \cite{Jiang}, \cite%
{Montaldo}, \cite{Oniciuc}, \cite{Oniciuc2}), to our knowledge, in Finsler
geometry, only harmonic maps have been considered, \cite{Mo}, \cite{Mo2}, 
\cite{Mo-book}, \cite{Mosel}, \cite{Nishikawa}, \cite{Shen-Zhang}. Still,
the rich potential of Finslerian geometric models makes us think that such a
study is at least necessary.

As a first step in this direction, we study in this paper biharmonic
mappings having as domain real Finsler spaces $(M,g)$ and as codomain,
Riemannian ones $(\tilde{M},\tilde{g})$ or, briefly, Finsler-to-Riemann
mappings. Our study continues the one made by Xiaohuan Mo and collaborators (%
\cite{Mo}, \cite{Mo2}, \cite{Mo-book}) for Finsler-to-Riemann harmonic
mappings.

First of all, we extend the concept of bienergy functional for
Finsler-to-Riemann mappings and determine its Euler-Lagrange equations,
i.e., the equations of Finsler-to-Riemann biharmonic maps. This process
points out a generalization of the rough Laplacian from Riemannian geometry.

Any Finsler-to Riemann harmonic map is biharmonic. Just as in Riemannian
geometry, there exist several cases in which the converse is also true; two
notable results in Riemannian geometry, due to Guoying Jiang, \cite{Jiang},
and C. Oniciuc, \cite{Oniciuc}, respectively, can be generalized without
difficulty to our situation:

1) Any biharmonic mapping whose domain is compact and boundaryless and whose
codomain has nonpositive sectional curvature, is harmonic.

2)\ Any biharmonic mapping whose codomain has strictly negative sectional
curvature, obeying the conditions: a) the norm of its tension is constant
and b)\ its rank is greater or equal to 2 at least at one point of its
domain, is harmonic.

Further, we study the biharmonicity of the identity map $id:(M,g)\rightarrow
(M,\tilde{g})$ in two cases of Finsler-to-Riemann transformations of metrics 
$g\mapsto \tilde{g},$ thus pointing out examples of nonharmonic biharmonic
maps. The second case, that of linearized perturbations, is inspired from
general relativity; even though we only consider here positive definite
metrics, in our opinion, it is illustrative.

In the last section, we determine the second variation of the bienergy
functional. Except for the facts that each of the expressions of the tension
and of the rough Laplacian gains an extra term and of the use of nonlinear
connections on $TM$, the first and second variation of the bienergy remain
formally similar to their Riemannian counterparts.

In the study of a Finsler space $(M,g)$, there are two major - and
equivalent - approaches: the one based on the tangent bundle $(TM,\pi ,M)$,
via horizontal lifts, and the one based on the pullback bundle $\pi ^{\ast
}TM.$ As noticed in \cite{Nishikawa}, the study of harmonic maps between
real Finsler manifolds is usually carried out on $\pi ^{\ast }TM$ (as in 
\cite{Mo}, \cite{Shen-Zhang}) while in the case of complex Finsler
manifolds, it relies on the geometry of $TM.$ In order to obtain a more
unified method, we preferred to work, also in the real case, on $TM$; the
geometric structures we used are the $TM$-correspondents of those in \cite%
{Mo}, \cite{Mo2}, \cite{Mo-book}.

\section{Biharmonic maps in Riemannian geometry}

In this section, we present in brief some results in \cite{Montaldo}, \cite%
{Jiang}, \cite{Balmus}.

Let $(M,g)$ and $(\tilde{M},\tilde{g})$ be two $\mathcal{C}^{\infty }$%
-smooth, connected Riemannian manifolds without boundary, of dimensions $n$
and $\tilde{n}$; unless elsewhere specified, we will assume, as in \cite%
{Jiang}, that $M$ is compact and orientable. On the two manifolds, we denote
the local coordinates by $(x^{i})_{i=\overline{1,n}},$ $(\tilde{x}^{\alpha
})_{\alpha =\overline{1,\tilde{n}}},$ the Levi-Civita connections by $\nabla
,$ $\tilde{\nabla}$ (with coefficients $\Gamma _{~jk}^{i},$ $\tilde{\Gamma}%
_{~\beta \gamma }^{\alpha }$) and by $\Gamma (E),$ $\Gamma (\tilde{E}),$ the
modules of $\mathcal{C}^{\infty }$-smooth sections of any vector bundles $E,%
\tilde{E}$ over $M$ and $\tilde{M}.$ Commas $_{,i}$ and $_{,\alpha }$ will
mean partial differentiation with respect to $x^{i}$ and $\tilde{x}^{\alpha
} $ and $\partial _{i},$ $\tilde{\partial}_{\alpha },$ the natural bases of
the modules $\Gamma (TM)$ and $\Gamma (T\tilde{M})$ respectively.

A $\mathcal{C}^{\infty }$-smooth mapping $\phi :M\rightarrow \tilde{M}$ is
called \textit{harmonic, }if it is a critical point of the \textit{energy
functional} 
\begin{equation}
E:\mathcal{C}^{\infty }(M,\tilde{M})\rightarrow \mathbb{R},~~E(\phi )=\dfrac{%
1}{2}\underset{M}{\int }\left\Vert d\phi \right\Vert ^{2}d\mathcal{V}_{g},
\label{energy_r}
\end{equation}%
where $d\phi $ is regarded as a section of the bundle $T^{\ast }M\otimes
\phi ^{-1}T\tilde{M},$ $\left\Vert d\phi \right\Vert ^{2}=trace_{g}(\phi
^{\ast }\tilde{g})=g^{ij}\tilde{g}_{\alpha \beta }\phi _{~,i}^{\alpha }\phi
_{~,j}^{\beta }$ is the squared Hilbert-Schmidt norm of $d\phi $ and $d%
\mathcal{V}_{g}$ is the Riemannian volume element on $M.$

Harmonic maps are solutions of the equation $\tau (\phi )=0,$ where, \cite%
{Montaldo},%
\begin{equation}
\tau (\phi )=g^{ij}\{\nabla _{\partial _{i}}^{\phi }d\phi (\partial
_{j})-d\phi (\nabla _{\partial _{i}}\partial _{j})\}=:g^{ij}(\nabla
_{\partial _{i}}^{\phi }d\phi )\partial _{j},  \label{tension_r_covar}
\end{equation}%
is a section of the bundle $\phi ^{-1}T\tilde{M},$ called the \textit{%
tension }of $\phi $ and $\nabla ^{\phi }$ is the connection induced by $%
\tilde{\nabla}$ in the pullback bundle $\phi ^{-1}T\tilde{M},$ \cite{Balmus}%
. In local writing:%
\begin{equation*}
\tau ^{\alpha }(\phi )=g^{ij}\{\phi _{~,ij}^{\alpha }+\tilde{\Gamma}_{~\beta
\gamma }^{\alpha }\phi _{~,i}^{\beta }\phi _{~,j}^{\gamma }-\Gamma
_{~ij}^{k}\phi _{~,k}^{\alpha }\}.
\end{equation*}

The above notion of harmonicity generalizes the usual one for mappings
between Euclidean spaces; notable examples\ include geodesic curves and
minimal Riemannian immersions.

\bigskip

\textit{Biharmonic maps} $\phi \in \mathcal{C}^{\infty }(M,\tilde{M})$ are
defined as critical points of the \textit{bienergy functional:}%
\begin{equation}
E_{2}(\phi )=\dfrac{1}{2}\underset{M}{\int }\left\langle \tau (\phi ),\tau
(\phi )\right\rangle d\mathcal{V}_{g};  \label{bienergy_r}
\end{equation}%
here $\left\langle ~~,~~\right\rangle $ denotes the scalar product on the
fibers of $T\tilde{M},$ determined by $\tilde{g}.$ The Euler-Lagrange
equation attached to the bienergy is, \cite{Montaldo}:%
\begin{equation}
\tau _{2}(\phi )=0,  \label{biharmonic_eqn_R}
\end{equation}%
where the \textit{bitension }$\tau _{2}(\phi )$ of $\phi $ is the section of
the bundle $\phi ^{-1}T\tilde{M}$ given by: 
\begin{equation}
\tau _{2}(\phi )=-\Delta ^{\phi }\tau (\phi )-trace_{g}(R^{\tilde{\nabla}%
}(d\phi ,\tau (\phi ))d\phi )  \label{bitension_r}
\end{equation}%
and the operator $\Delta ^{\phi }=-g^{ij}(\nabla _{\partial _{i}}^{\phi
}\nabla _{\partial _{j}}^{\phi }-\nabla _{\nabla _{\partial _{i}}\partial
_{j}}^{\phi })$ is the rough Laplacian, acting on sections of $\phi ^{-1}T%
\tilde{M}$.

\textbf{Remarks: }1) Equation (\ref{biharmonic_eqn_R}) is the Riemannian
generalization of the biharmonic equation in Euclidean spaces, \cite%
{Selvadurai}.

2)\ Any harmonic map $\phi :M\rightarrow \tilde{M}$ is biharmonic.

\section{Finsler structures}

In the following, except for the metric structure on $M$ (and related
quantities) we preserve the notations and conventions in Section 2. We
denote by $TM$ and $T\tilde{M}$ the tangent bundles of the manifolds $M$ and 
$\tilde{M}$ and their local coordinates, by $(x,y):=(x^{i},y^{i}),$ $(\tilde{%
x},\tilde{y}):=(\tilde{x}^{\alpha },\tilde{y}^{\alpha });$ dots $_{\cdot i}$
and $_{\cdot \alpha }$ will mean partial differentiation with respect to $%
y^{i}$ and $\tilde{y}^{\alpha }.$

\bigskip

\textbf{A. Metric structure: }A \textit{Finsler structure}, \cite{Shen}, 
\cite{Mo-book}, on the manifold $M$ is a function $F:TM\rightarrow \mathbb{R}
$ with the properties:

1) $F(x,y)\ $is $\mathcal{C}^{\infty }$-smooth for $y\not=0$ and continuous
at $y=0.$

2)$\ F(x,\lambda y)=\lambda F(x,y),$ $\forall \lambda >0$;

3)\ The \textit{Finslerian metric tensor}:%
\begin{equation}
g_{ij}(x,y):=\dfrac{1}{2}(F^{2}(x,y))_{\cdot ij}  \label{metric_tensor_F}
\end{equation}%
is positive definite.

The arc length of a curve $c$ on the Finsler space $(M,g)$ is given, \cite%
{Shen}, by:%
\begin{equation}
l(c)=~\underset{c}{\int }F(x,dx).  \label{arc_length}
\end{equation}%
Condition 2) above insures the independence of $l(c)$ of the chosen
parametrization of $c.$

\bigskip

\textbf{B. Nonlinear connection and adapted frame on }$TM:$ Ehresmann (or 
\textit{nonlinear, }\cite{Shen}) connections $TTM=HTM\oplus VTM,$ with $%
VTM=Span(\dfrac{\partial }{\partial y^{i}})$ help simplify computations in
Finsler geometry and obtain geometric objects with simple transformation
rules. A typical choice is the \textit{Cartan nonlinear connection}, built
as follows, \cite{Shen}.

Geodesics of the Finsler space $(M,g)$ are defined as critical points $c$ of
the arc length (\ref{arc_length}); in the natural parametrization, their
equations are:%
\begin{equation}
\dfrac{dy^{i}}{ds}+2G^{i}(x,y)=0,~\ y=\dot{x};  \label{spray_coeff}
\end{equation}%
with $2G^{i}(x,y)=\dfrac{1}{2}g^{ih}\left( (F^{2})_{\cdot
h,k}y^{k}-(F^{2})_{,h}\right) ;$ this defines the local coefficients $%
G_{~j}^{i}=G_{~j}^{i}(x,y)$ of the Cartan nonlinear connection as: 
\begin{equation}
G_{~j}^{i}:=G_{~\cdot j}^{i}.  \label{Cartan_conn}
\end{equation}%
The Cartan nonlinear connection gives rise to the adapted basis: 
\begin{equation}
(\delta _{i}=\dfrac{\partial }{\partial x^{i}}-G_{~i}^{j}(x,y)\dfrac{%
\partial }{\partial y^{j}},~~\ \ \dot{\partial}_{i}=\dfrac{\partial }{%
\partial y^{i}})  \label{general_adapted_basis}
\end{equation}%
on $\Gamma (TTM)$ and to its dual $(dx^{i},~\ \delta
y^{i}=dy^{i}+G_{~j}^{i}dx^{j}).$

With respect to coordinate transformations on $TM,$ induced by coordinate
transformations $x^{i^{\prime }}=x^{i^{\prime }}(x)$ on $M,$ the elements of
the adapted basis/cobasis transform by the same rules as vector/covector
fields on $M$, \cite{Lagrange}.

Any vector field $X$ on $TM$ can be decomposed as: $X=hX+vX,$ $\
hX=X^{i}\delta _{i},$ $vX=\hat{X}^{i}\dot{\partial}_{i};$ its \textit{%
horizontal component} $hX$ and its \textit{vertical component} $vX$ are
vector fields on $TM$. This leads to a simple rule of transformation for $%
X^{i},\hat{X}^{i}.$ A similar situation holds for 1-forms $\omega =h\omega
+v\omega ,$ $h\omega =\omega _{i}dx^{i},$ $v\omega =\hat{\omega}_{i}\delta
y^{i}$, \cite{Lagrange}) and more generally, for tensors of any rank on $TM$.

Using Cartan nonlinear connection, tangent vector fields to lifts $c^{\prime
}:=(c,\dot{c})$ to $TM$ of unit speed geodesics of $M$ are always
horizontal, \cite{Shen}, \cite{Lagrange}.

\bigskip

The adapted basis $\{\delta _{i},\dot{\partial}_{i}\}$ is generally
nonholonomic, the Lie brackets of its elements are:%
\begin{equation*}
\lbrack \delta _{j},\delta _{k}]=R_{~jk}^{i}(x,y)\dot{\partial}%
_{i},~~~~[\delta _{j},\dot{\partial}_{k}]=G_{~jk}^{i}(x,y)\dot{\partial}%
_{i},~\ \ [\dot{\partial}_{j},\dot{\partial}_{k}]=0;
\end{equation*}%
where: 
\begin{equation}
R_{~jk}^{i}(x,y)=\delta _{k}G_{~j}^{i}-\delta _{j}G_{~k}^{i},~\ \ \
G_{~jk}^{i}(x,y):=G_{~\cdot j\cdot k}^{i}(x,y).  \label{Lie_bracket_comps}
\end{equation}

\bigskip

\textbf{C. }As \textbf{covariant differentiation }rule on $TM,$ we will use
the one given by the \textit{Chern-Rund affine connection} $D$ on $TM,$ \cite%
{Bucataru}, locally described by:%
\begin{equation}
D_{\delta _{k}}\delta _{j}=~\Gamma _{~jk}^{i}\delta _{i},~\ D_{\delta _{k}}%
\dot{\partial}_{j}=\Gamma _{~jk}^{i}\dot{\partial}_{i},~\ D_{\dot{\partial}%
_{k}}\delta _{j}=~D_{\dot{\partial}_{k}}\dot{\partial}_{j}=0,
\label{chern_conn}
\end{equation}%
where~\ $\Gamma _{~jk}^{i}=\dfrac{1}{2}g^{ih}(\delta _{k}g_{hj}+\delta
_{j}g_{hk}-\delta _{h}g_{jk})$ are the "adapted" Christoffel symbols%
\footnote{%
In the Finslerian case, the usual Christoffel symbols $\gamma _{~jk}^{i}=%
\dfrac{1}{2}g^{ih}(g_{hj,k}+g_{hk,j}-g_{jk,h})$ do \textit{not }generally
represent the coefficients of an affine connection on $TM.$} of $g.$ The
Chern-Rund connection preserves by parallelism the horizontal and vertical
distributions on $TTM$, i.e.,%
\begin{equation*}
D_{X}(hY)=h~D_{X}Y,~~~D_{X}(vY)=v~D_{X}Y;
\end{equation*}%
it is generally, only $h$-metrical: 
\begin{equation*}
D_{hX}g=0,~\forall X\in \Gamma (TTM).
\end{equation*}

The Chern-Rund connection $D$ has nontrivial torsion:%
\begin{equation}
T=R_{~jk}^{i}\dot{\partial}_{i}\otimes dx^{k}\otimes dx^{j}+P_{~jk}^{i}\dot{%
\partial}_{i}\otimes \delta y^{k}\otimes dx^{j},  \label{torsion}
\end{equation}%
with $R_{~jk}^{i}$ as in (\ref{Lie_bracket_comps})\ and $%
P_{~jk}^{i}=G_{~jk}^{i}-\Gamma _{~jk}^{i};$ the latter defines a horizontal
1-form:%
\begin{equation}
P=P_{i}dx^{i},~\ P_{i}:=P_{~ij}^{j},  \label{torsion_comps}
\end{equation}%
which will be used in the following. We notice that the torsion of $D$ has
only vertical components:%
\begin{equation}
hT(X,Y)=0,~\ \forall X,Y\in \Gamma (TTM).  \label{torsion_property}
\end{equation}

The curvature $R$ of $D$ is locally described by:%
\begin{eqnarray}
R &=&R_{j~kl}^{~i}dx^{l}\otimes dx^{k}\otimes dx^{j}\otimes \delta
_{i}+R_{j~kl}^{~i}dx^{l}\otimes dx^{k}\otimes \delta y^{j}\otimes \dot{%
\partial}_{i}+;  \label{curvature_comps} \\
&&+P_{j~kl}^{~i}\delta y^{l}\otimes dx^{k}\otimes dx^{j}\otimes \delta
_{i}+P_{j~kl}^{~i}\delta y^{l}\otimes dx^{k}\otimes \delta y^{j}\otimes \dot{%
\partial}_{i},  \notag
\end{eqnarray}%
where $R_{j~kl}^{~i}=\delta _{l}\Gamma _{~jk}^{i}-\delta _{k}\Gamma
_{~jl}^{i}+\Gamma _{~jk}^{h}\Gamma _{~hl}^{i}-\Gamma _{~jl}^{h}\Gamma
_{~hk}^{i}$ and$\ P_{j~kl}^{~i}=\Gamma _{~jk\cdot l}^{i}.$

\bigskip

We consider the above notions also for the Riemannian manifold $(\tilde{M},%
\tilde{g})$ and designate them by tildes. In this case: $\tilde{\Gamma}%
_{~\beta \gamma }^{\alpha }=\tilde{G}_{~\beta \gamma }^{\alpha }=\tilde{%
\gamma}_{~\beta \gamma }^{\alpha }$, that is: 
\begin{equation}
(\tilde{\nabla}_{X}Y)^{\tilde{h}}=\tilde{D}_{X^{\tilde{h}}}Y^{\tilde{h}},~\
\forall X,Y\in \Gamma (T\tilde{M}),  \label{Levi-Civita_lift}
\end{equation}%
where the superscript $^{\tilde{h}}$ indicates the horizontal of vector
fields from $\tilde{M}$ to $T\tilde{M};$\ also, $\tilde{G}_{~\beta }^{\alpha
}(\tilde{x},\tilde{y})=\tilde{\gamma}_{~\beta \gamma }^{\alpha }(\tilde{x})%
\tilde{y}^{\gamma },$ $\tilde{P}_{~\beta \gamma }^{\alpha }=0;$ the
Chern-Rund connection becomes "fully"\ metrical:%
\begin{equation}
\tilde{D}_{X}\tilde{g}=0,~\forall X\in \Gamma (TT\tilde{M}).
\label{full_metricity}
\end{equation}%
The only nonzero local component of curvature tensor $\tilde{R}$ is $\tilde{R%
}_{\beta ~\gamma \delta }^{~\alpha },$ i.e.:%
\begin{equation}
\tilde{R}(X,Y)Z=\tilde{R}(\tilde{h}X,\tilde{h}Y)Z,~\ \ \ \ \ \forall
X,Y,Z\in \Gamma (TT\tilde{M});  \label{Riem_R}
\end{equation}%
$\tilde{R}_{\beta ~\gamma \delta }^{~\alpha }$ coincide with the components
of the curvature $R^{\tilde{\nabla}}$ of the Levi-Civita connection of $%
\tilde{g}$ and are thus subject to the same symmetries; Bianchi identities, 
\cite{Shen}, also acquire the same form as those of $R^{\tilde{\nabla}}.$
Ricci identities of $\tilde{D}$ take the local form:%
\begin{equation}
\begin{array}{l}
\tilde{D}_{\delta _{\rho }}\tilde{D}_{\delta _{\gamma }}Z^{\alpha }-\tilde{D}%
_{\delta _{\gamma }}\tilde{D}_{\delta _{\rho }}Z^{\alpha }=\tilde{R}_{\beta
~\gamma \rho }^{~\alpha }Z^{\beta }+\tilde{R}_{~\gamma \rho }^{\beta
}Z_{~\cdot \beta }^{\alpha }; \\ 
\ \tilde{D}_{\dot{\partial}_{\rho }}\tilde{D}_{\delta _{\gamma }}Z^{\alpha }-%
\tilde{D}_{\delta _{\gamma }}\tilde{D}_{\dot{\partial}_{\rho }}Z^{\alpha
}=0,~\ \ ~\ \ \ \ \forall Z=Z^{\alpha }\delta _{\alpha }\in \Gamma (HT\tilde{%
M}).\ 
\end{array}
\label{Ricci_R}
\end{equation}

Another useful property of $\tilde{D}$ is:%
\begin{equation}
\tilde{D}_{\tilde{\delta}_{\beta }}\tilde{y}^{\alpha }=0.  \label{deflection}
\end{equation}

The Riemannian metric $\tilde{g}$ gives rise to a scalar product on the
fibers on $HT\tilde{M},$ which will be denoted by $\left\langle
~~,~~\right\rangle :$%
\begin{equation}
\left\langle X,Y\right\rangle =\tilde{g}_{\alpha \beta }X^{\alpha }Y^{\beta
},~\ \ ~\ \forall X=X^{\alpha }\tilde{\delta}_{\alpha },Y=Y^{\alpha }\tilde{%
\delta}_{\alpha }\in \Gamma (HT\tilde{M}).  \label{scalar_product}
\end{equation}

\bigskip

\textbf{D. Volume form and integration domain. }Consider the Riemannian
volume element $d\mathcal{V}_{g}=\sqrt{\det G}dx\wedge \delta y=\det
gdx\wedge \delta y$ on $TM,$ determined by the \textit{Sasaki lift, }\cite%
{Shen}, $G:=g_{ij}dx^{i}\otimes dx^{j}+g_{ij}\delta y^{i}\otimes \delta
y^{j} $ of $g$ to $TM$. For an $x,y$-dependent function $\alpha
:TM\rightarrow \mathbb{R},$ we will consider its integral on the total space
of the unit ball bundle of $M$ (which is compact, since $g$ is positive
definite), divided by the volume $Vol\mathbb{B}^{n}$ of the unit ball in the
Euclidean space $\mathbb{R}^{n}$, i.e.,%
\begin{equation}
\underset{BM}{\int }\alpha (x,y)d\mathcal{V}_{g}=\dfrac{1}{Vol\mathbb{B}^{n}}%
\underset{M}{\int }(\underset{B_{x}}{\int }\alpha (x,y)\det g(x,y)dy)dx,
\end{equation}%
where $B_{x}=\left\{ y\in T_{x}M~|~g_{ij}y^{i}y^{j}\leq 1\right\} $ (on $M,$
this construction provides a generalization of the Riemannian volume
element, called the \textit{Holmes-Thompson volume element, }\cite{Shen1}).

The divergence $divX=\dfrac{1}{\det g}\delta _{i}(X^{i}\det
g)-G_{~ij}^{j}X^{i},$ of a horizontal vector field $X=X^{i}\delta _{i}$ on $%
TM,$ \cite{Zhong}, can be expressed in terms of Chern-Rund covariant
derivatives as:%
\begin{equation*}
divX=~D_{\delta _{i}}X^{i}-P(X);
\end{equation*}%
we will also use this relation in the form:%
\begin{equation}
g^{ij}\delta _{i}X_{j}=divX+g^{ij}\Gamma _{~ij}^{k}X_{k}+P_{i}X^{i}.
\label{divergence}
\end{equation}

\section{Some remarks on Finsler-to-Riemann maps}

Let $\phi :M\rightarrow \tilde{M},$ $(x^{i})\mapsto (\phi ^{\alpha }(x^{i}))$
be $\mathcal{C}^{\infty }$-smooth. Between the tangent bundles $TM$ and $T%
\tilde{M}$, it acts the differential $\Phi :=d\phi $ (regarded as a mapping
between manifolds); throughout this section, we will use alternatively the
two notations $\Phi $ and $d\phi .$ The connection $\tilde{D}$ determines a
connection $D^{d\phi }$ in the pullback bundle $d\phi ^{-1}(TT\tilde{M})$: 
\begin{equation}
D_{X}^{d\phi }(\Phi ^{\ast }Y):=\tilde{D}_{\Phi _{\ast }X}Y,~~~~\forall X\in
\Gamma (TTM),~Y\in \Gamma (TT\tilde{M});  \label{pullback_conn_TM}
\end{equation}%
with $(\Phi ^{\ast }Y)_{(x,y)}=Y_{\Phi (x,y)}$; hence,%
\begin{equation}
D_{X}^{d\phi }(\tilde{h}Z)=\tilde{h}D_{X}^{d\phi }Z,~\ \ D_{X}^{d\phi }(%
\tilde{v}Z)=\tilde{v}D_{X}^{d\phi }Z,~\ \ \ \forall Z\in \Gamma (d\phi
^{-1}(TT\tilde{M})).  \label{d-connection}
\end{equation}

The mapping $\Phi $ is locally described by: 
\begin{equation*}
\Phi :~\ \ \tilde{x}^{\alpha }=\phi ^{\alpha }(x),~\ \tilde{y}^{\alpha
}=\phi _{~,j}^{\alpha }(x)y^{j}.
\end{equation*}%
With $\phi ^{\alpha ^{\prime }}(x,y):=\phi _{~,j}^{\alpha }(x)y^{j},$ we
have, for $X\in \Gamma (HTM):$ $\Phi _{\ast }X=X(\phi ^{\alpha })\partial
_{\alpha }+X(\phi ^{\alpha ^{\prime }})\dot{\partial}_{\alpha }.$ Taking $%
X=X^{i}\delta _{i}$ and expressing $\Phi _{\ast }X$ in the adapted bases,%
\begin{equation}
\Phi _{\ast }X=X^{i}\phi _{~,i}^{\alpha }\tilde{\delta}_{\alpha }+{\large (}%
X^{i}\delta _{i}\phi ^{\alpha ^{\prime }}+\tilde{N}_{~\beta }^{\alpha
}X(\phi ^{\beta }){\large )}\dot{\partial}_{\alpha }.  \label{image_Phi}
\end{equation}%
The horizontal component $\tilde{h}\Phi _{\ast }X$ will have a peculiar
importance.

\begin{lemma}
For any horizontal vector field $X=X^{i}\delta _{i}$ on $TM:$%
\begin{equation}
\tilde{h}\Phi _{\ast }X=X^{i}\phi _{~,i}^{\alpha }\delta _{\alpha }=:d\phi ^{%
\tilde{h}}(X),  \label{notation_horiz_comp_differential}
\end{equation}%
where $d\phi ^{\tilde{h}}:=\phi _{\,,i}^{\alpha }dx^{i}\otimes \delta
_{\alpha }$ is the horizontal lift of the vector-valued 1-form $d\phi $.
\end{lemma}

Consider a 1-parameter variation $f:I_{\varepsilon }\times
M,~~f=f(\varepsilon ,x),$ $f(0,x)=\phi (x)$ of $\phi $ and:%
\begin{equation*}
F:=df:T(I_{\varepsilon }\times M)\rightarrow T\tilde{M}.
\end{equation*}

On $T(I_{\varepsilon }\times M),$ the local coordinates are $(\varepsilon
,x,\varepsilon ^{\prime },y).$ Taking on the interval $I_{\varepsilon
}\subset \mathbb{R},$ the Euclidean metric and the product Finsler metric on 
$I_{\varepsilon }\times M,$ we will obtain a trivial prolongation of the
Cartan nonlinear connection to this new manifold, which produces the adapted
basis $\{\partial _{\varepsilon },\delta _{i},\partial _{\varepsilon
^{\prime }},\dot{\partial}_{i}\}$ and a trivial prolongation of the Chern
connection $D$ (which we will denote again by $D$). i.e., $D_{\partial
_{\varepsilon }}\delta _{i}=D_{\partial _{\varepsilon }}\dot{\partial}%
_{i}=0,~\ D_{\delta _{i}}\partial _{\varepsilon }=D_{\dot{\partial}%
_{i}}\partial _{\varepsilon }=0$ etc. We also notice that $[\partial
_{\varepsilon },\delta _{i}]=0$ and $[\partial _{\varepsilon },\dot{\partial}%
_{i}]=0.$

The connection $D^{df}$ will be prolonged to $df^{-1}(TT\tilde{M}),$ by: 
\begin{equation}
D_{\partial _{\varepsilon }}^{df}(F^{\ast }X):=~\tilde{D}_{F_{\ast }\partial
_{\varepsilon }}X,~\ \ X\in \Gamma (TT\tilde{M}).  \label{pullback_conn}
\end{equation}

\begin{lemma}
For any $X,Y\in \Gamma (TTM)$, $Z\in \Gamma (df^{-1}(TT\tilde{M})),$ $%
p=(x,y)\in TM$: 
\begin{equation}
D_{\partial _{\varepsilon }}^{df}(df^{\tilde{h}}(X))=D_{X}^{df}(df^{\tilde{h}%
}(\partial _{\varepsilon })).  \label{commutation_rule_2}
\end{equation}
\end{lemma}

\begin{proof}
(\ref{torsion_property}) says that $0=$ $\tilde{h}\tilde{T}(F_{\ast
}\partial _{\varepsilon },F_{\ast }X)=$ $h\{D_{\partial _{\varepsilon
}}^{df}(F_{\ast }X)-D_{X}^{df}(F_{\ast }\partial _{\varepsilon })-[F_{\ast
}X,F_{\ast }\partial _{\varepsilon }]\}.$ The result follows then from (\ref%
{d-connection})\ and $[X,\partial _{\varepsilon }]=0$.
\end{proof}

\section{Bienergy and its first variation}

The energy of a Finsler-to-Riemann mapping $\phi :M\rightarrow \tilde{M},$ 
\cite{Mo}, is $E(\phi )=\dfrac{1}{2}\underset{BM}{\int }g^{ij}\tilde{g}%
_{\alpha \beta }\phi _{~,i}^{\alpha }\phi _{~,j}^{\beta }d\mathcal{V}_{g},$
(note that $\phi =\phi (x),$ hence $\phi _{~,i}^{\alpha }=\delta _{i}\phi
^{\alpha }$), or, in our language: 
\begin{equation*}
E(\phi )=\dfrac{1}{2}\underset{BM}{\int }g^{ij}\left\langle d\phi ^{\tilde{h}%
}(\delta _{i}),d\phi ^{\tilde{h}}(\delta _{i})\right\rangle d\mathcal{V}_{g},
\end{equation*}%
with $\left\langle ~~,~~\right\rangle $ as in (\ref{scalar_product}). The 
\textit{tension} of $\phi :M\rightarrow \tilde{M}$, \cite{Mo}, can be
regarded as a section of $(d\phi )^{-1}HT\tilde{M}$:%
\begin{equation}
\tau (\phi )=g^{ij}\{D_{\delta _{i}}^{d\phi }d\phi ^{\tilde{h}}(\delta
_{j})-d\phi ^{\tilde{h}}(D_{\delta _{i}}\delta _{j})-P_{i}d\phi ^{\tilde{h}%
}(\delta _{j})\}.  \label{tension_FR}
\end{equation}%
The mapping $\phi $ is \textit{harmonic} iff its tension vanishes
identically.

\bigskip

It appears as natural to define the \textit{bienergy }of a mapping $\phi
:M\rightarrow \tilde{M}$ as:%
\begin{equation}
E_{2}(\phi )=\dfrac{1}{2}\underset{BM}{\int }\left\langle \tau (\phi ),\tau
(\phi )\right\rangle d\mathcal{V}_{g}.  \label{bienergy}
\end{equation}
Accordingly, by a Finsler-to-Riemann \textit{biharmonic map} we will mean a
critical point of the bienergy (\ref{bienergy}).

In order to determine the critical points of $E_{2}$, we take variations $%
f=f(\varepsilon ,x)$ of $\phi $ as above and denote by%
\begin{equation}
\mathbf{V}:=(df(\partial _{\varepsilon }))^{\tilde{h}}=df^{\tilde{h}%
}(\partial _{\varepsilon }^{~\tilde{h}}),~\ \ \ \ \ \ V:=\mathbf{V}%
_{|\varepsilon =0.}  \label{variation_vector_field}
\end{equation}%
the horizontal lift of the associated deviation vector field $df(\partial
_{\varepsilon })$.

Since the Chern-Rund connection on the codomain $T\tilde{M}$ is metrical,%
\begin{equation*}
\dfrac{dE_{2}}{d\varepsilon }(f)=\dfrac{1}{2}\dfrac{d}{d\varepsilon }%
\underset{BM}{\int }\left\langle \tau (f),\tau (f)\right\rangle d\mathcal{V}%
_{g}=\underset{BM}{\int }\left\langle D_{\partial _{\varepsilon }}^{df}\tau
(f),\tau (f)\right\rangle d\mathcal{V}_{g}.
\end{equation*}

Let us evaluate the term $D_{\partial _{\varepsilon }}^{df}\tau (f):$

\begin{eqnarray*}
&&D_{\partial _{\varepsilon }}^{df}\tau (f)=D_{\partial _{\varepsilon
}}^{df}\{g^{ij}(D_{\delta _{i}}^{df}(df^{\tilde{h}}(\delta _{j}))-df^{\tilde{%
h}}(D_{\delta _{i}}\delta _{j})-P_{i}df^{\tilde{h}}(\delta _{j}))\}= \\
&=&g^{ij}\left\{ D_{\partial _{\varepsilon }}^{df}D_{\delta _{i}}^{df}(df^{%
\tilde{h}}(\delta _{j}))-D_{\partial _{\varepsilon }}^{df}(df^{\tilde{h}%
}(D_{\delta _{i}}\delta _{j}))-D_{\partial _{\varepsilon }}^{df}(P_{i}df^{%
\tilde{h}}(\delta _{j}))\right\}
\end{eqnarray*}%
($g^{ij}$ can be taken in front of the $\partial _{\varepsilon }$%
-derivative, since in $g^{ij}=g^{ij}(x,y)$ the coordinates $x,y$ do not
depend on $\varepsilon $). Commuting derivatives by means of the curvature
tensor of $\tilde{D},$ taking (\ref{Riem_R}) and $[\delta _{i},\partial
_{\varepsilon }]=0$ into account, 
\begin{equation*}
D_{\partial _{\varepsilon }}^{df}D_{\delta _{i}}^{df}(df^{\tilde{h}}(\delta
_{j}))=\tilde{R}(\mathbf{V},df^{\tilde{h}}(\delta _{i}))df^{\tilde{h}%
}(\delta _{j})+D_{\delta _{i}}^{df}D_{\partial _{\varepsilon }}^{df}(df^{%
\tilde{h}}(\delta _{j})).
\end{equation*}%
By (\ref{commutation_rule_2}), the term $D_{\delta _{i}}^{df}D_{\partial
_{\varepsilon }}^{df}(df^{\tilde{h}}(\delta _{j}))$ becomes $D_{\delta
_{i}}^{df}D_{\delta _{j}}^{df}\mathbf{V.}$ Using (\ref{commutation_rule_2})
also in the expression $D_{\partial _{\varepsilon }}^{df}(df^{\tilde{h}%
}(D_{\delta _{i}}\delta _{j}))-D_{\partial _{\varepsilon }}^{df}(P_{i}df^{%
\tilde{h}}(\delta _{j}))$ and summing up, we get:%
\begin{eqnarray}
D_{\partial _{\varepsilon }}^{df}\tau (f) &=&g^{ij}\{\tilde{R}(\mathbf{V}%
,df^{\tilde{h}}(\delta _{i}))df^{\tilde{h}}(\delta _{j})+D_{\delta
_{i}}^{df}D_{\delta _{j}}^{df}\mathbf{V}\mathcal{-}  \label{derivative_tau}
\\
&&-D_{D_{\delta _{i}}\delta _{j}}^{df}\mathbf{V}-P_{i}D_{\delta _{j}}^{df}%
\mathbf{V}\}.  \notag
\end{eqnarray}

We notice the operators%
\begin{equation}
g^{ij}(-D_{\delta _{i}}^{df}D_{\delta _{j}}^{df}+D_{D_{\delta _{i}}\delta
_{j}}^{df}+P_{i}D_{\delta _{j}}^{df})=:\Delta ^{df},~\mathcal{J}=-\Delta
^{df}-trace_{g}\tilde{R}(df^{\tilde{h}},~\cdot ~)df^{\tilde{h}},
\label{rough_Laplacian}
\end{equation}%
acting on sections of the bundle $(df)^{-1}(HT\tilde{M})$. With this, we
have: 
\begin{equation}
D_{\partial _{\varepsilon }}^{df}\tau (f)=-\Delta ^{df}\mathbf{V}-g^{ij}%
\tilde{R}(df^{\tilde{h}}(\delta _{i}),\mathbf{V})df^{\tilde{h}}(\delta _{j})=%
\mathcal{J}(\mathbf{V}).  \label{deriv_tau_1}
\end{equation}%
Evaluating at $\varepsilon =0$ and substituting into the expression of the
variation, 
\begin{equation}
\dfrac{dE_{2}}{d\varepsilon }(f)|_{\varepsilon =0}=\underset{BM}{\int }%
\left\langle \mathcal{J}(V),\tau (\phi )\right\rangle d\mathcal{V}_{g}.
\label{first_variation_interm}
\end{equation}%
It remains to transform the above expression so as to have $V$ in the right
hand side of the scalar product. This will be easy using the following lemma.

\begin{lemma}
\label{Lemma1} The operators $\Delta ^{d\phi }$ and $\mathcal{J}$ are
self-adjoint:%
\begin{equation}
\underset{BM}{\int }\left\langle \Delta ^{d\phi }X,Y\right\rangle d\mathcal{V%
}_{g}=\underset{BM}{\int }\left\langle X,\Delta ^{d\phi }Y\right\rangle d%
\mathcal{V}_{g},\ \underset{BM}{\int }\left\langle \mathcal{J}%
X,Y\right\rangle d\mathcal{V}_{g}=\underset{BM}{\int }\left\langle X,%
\mathcal{J}Y\right\rangle d\mathcal{V}_{g}.  \label{self-adj}
\end{equation}%
for any $X,Y\in \Gamma (d\phi ^{-1}(HT\tilde{M})).$
\end{lemma}

\begin{proof}
We start from the left hand side of the first relation (\ref{self-adj});
integrating by parts the term $\underset{BM}{\int }\left\langle
-g^{ij}D_{\delta _{i}}^{df}D_{\delta _{j}}^{df}X,Y\right\rangle d\mathcal{V}%
_{g}$ and applying (\ref{divergence}), we get:%
\begin{equation}
\underset{BM}{\int }\left\langle \Delta ^{d\phi }X,Y\right\rangle d\mathcal{V%
}_{g}=-\underset{BM}{\int }g^{ij}\left\langle D_{~\delta _{i}}^{d\phi
}X,~D_{~\delta _{j}}^{d\phi }Y\right\rangle d\mathcal{V}_{g}.
\label{adjoint1}
\end{equation}%
Integrating once again by parts, we obtain (\ref{self-adj}). The
self-adjointness of $\mathcal{J}$ follows then from the symmetries of $%
\tilde{R}.$
\end{proof}

The operator $\Delta ^{d\phi }$ is a generalization of the rough Laplacian
from Riemannian geometry, built in the same spirit as the horizontal
Laplacian acting on differential forms in \cite{Zhong}, \cite{Zhong1}.

Using Lemma \ref{Lemma1} in (\ref{first_variation_interm}), we get:

\begin{proposition}
a) The first variation of the bienergy of a mapping $\phi :M\rightarrow 
\tilde{M}$ from the Finsler space $(M,g)$ to the Riemann space $(\tilde{M},%
\tilde{g})$ is:%
\begin{equation}
\dfrac{dE_{2}(f)}{d\varepsilon }|_{\varepsilon =0}=\underset{BM}{\int }%
\left\langle -\Delta ^{d\phi }\tau (\phi )-trace_{g}\tilde{R}(d\phi ^{\tilde{%
h}},\tau (\phi ))d\phi ^{\tilde{h}},V\right\rangle d\mathcal{V}_{g};
\label{first_var_bienergy}
\end{equation}

b) The mapping $\phi $ is biharmonic iff: 
\begin{equation}
\tau _{2}(\phi ):=-\Delta ^{d\phi }\tau (\phi )-trace_{g}\tilde{R}(d\phi ^{%
\tilde{h}},\tau (\phi ))d\phi ^{\tilde{h}}=0.
\label{biharmonic_Finsler_to_Riemann}
\end{equation}
\end{proposition}

\textbf{Remarks. }1) In the above, we considered, as in \cite{Jiang}, that $%
M $ is compact and without boundary. Elsewhere, all the discussion can be
made on a compact subset $\mathcal{D}$ of $M$; in this case, we assume that,
on the boundary of $\mathcal{D}$, the vector field $V$ and the covariant
derivatives $D_{\delta _{i}}^{d\phi }V$ vanish.

2) Any harmonic map from a Finsler space to a Riemann one is biharmonic,
namely, a minimum point for the bienergy functional. A biharmonic map which
is not harmonic will be called \textit{proper biharmonic.}

\textbf{Particular cases:}

1)\ If $\tilde{M}=\mathbb{R}^{n}$ with the Euclidean metric, then the
biharmonic equation (\ref{biharmonic_Finsler_to_Riemann}) becomes:%
\begin{equation*}
\Delta ^{d\phi }\tau (\phi )=0.
\end{equation*}

2)\ If $\tilde{M}=\mathbb{S}^{n}$ is the unit Euclidean sphere, then, using
the expression of the Riemann tensor of a space form, we get that $\phi
:M\rightarrow \mathbb{S}^{n}$ is biharmonic iff:%
\begin{equation*}
\Delta ^{d\phi }\tau (\phi )+2e(\phi )\tau (\phi )-trace_{g}\left\langle
d\phi ^{\tilde{h}},\tau (\phi )\right\rangle d\phi ^{\tilde{h}}=0,
\end{equation*}%
where $e(\phi )=\dfrac{1}{2}trace_{g}\left\langle d\phi ^{\tilde{h}},d\phi ^{%
\tilde{h}}\right\rangle $ is the energy density of $\phi .$ The result is
similar to the one in the Riemannian case, \cite{Balmus}.

3)\ If $M$ is a \textit{weakly Landsberg }manifold, i.e., \cite{Mo}, if $%
P=0, $ then the expressions of the tension and of the rough Laplacian become
formally similar to the ones in the Riemannian case - just, depending on the
fiber coordinates $y^{i}$: $\tau (\phi )=trace_{g}D^{d\phi }(d\phi ^{\tilde{h%
}}),~\ \ \Delta ^{df}=g^{ij}(-D_{\delta _{i}}^{df}D_{\delta
_{j}}^{df}+D_{D_{\delta _{i}}\delta _{j}}^{df}).$

\section{Existence of proper biharmonic maps}

The following two results represent generalizations to Finsler-to-Riemann
maps of two theorems in \cite{Jiang} and \cite{Oniciuc} respectively.

\begin{theorem}
If $(M,g)$ is a compact Finslerian manifold without boundary and $(\tilde{M},%
\tilde{g})$ is Riemannian with nonpositive sectional curvature, then any
biharmonic map $\phi :M\rightarrow \tilde{M}$ is harmonic.
\end{theorem}

\begin{proof}
The proof follows similar steps to the one in the Riemannian case, \cite%
{Jiang}. We apply the horizontal Laplace-Beltrami operator, \cite{Zhong}, $%
\Delta f:=-div(grad_{h}~f),$ where $grad_{h}f:=(g^{ij}\delta _{j}f)\delta
_{i},$ to the scalar function $f:=\left\Vert \tau (\phi )\right\Vert ^{2},$
defined on $TM:$%
\begin{equation*}
\begin{array}{c}
-\dfrac{1}{2}\Delta \left\Vert \tau (\phi )\right\Vert ^{2}=\dfrac{1}{2}%
\{D_{\delta _{i}}(g^{ij}\delta _{j}\left\Vert \tau (\phi )\right\Vert
^{2})-g^{ij}P_{i}\delta _{j}\left\Vert \tau (\phi )\right\Vert ^{2}\}= \\ 
=\dfrac{1}{2}g^{ij}\{\delta _{i}\delta _{j}\left\Vert \tau (\phi
)\right\Vert ^{2}-\Gamma _{~ij}^{k}\delta _{k}\left\Vert \tau (\phi
)\right\Vert ^{2}-P_{i}\delta _{j}\left\Vert \tau (\phi )\right\Vert ^{2}\}.%
\end{array}%
\end{equation*}

Taking into account that $\Gamma _{~ij}^{k}\delta _{k}=D_{\delta _{i}}\delta
_{j}$ and expressing the action of the adapted basis vector fields $\delta
_{i},\delta _{j}$ on $\left\Vert \tau (\phi )\right\Vert ^{2}=\left\langle
\tau (\phi ),\tau (\phi )\right\rangle $ in terms of $D^{d\phi }$-covariant
derivatives, we obtain:%
\begin{equation}
-\dfrac{1}{2}\Delta \left\Vert \tau (\phi )\right\Vert ^{2}=-\left\langle
\Delta ^{d\phi }\tau (\phi ),\tau (\phi )\right\rangle +g^{ij}\left\langle
D_{\delta _{i}}^{d\phi }\tau (\phi ),D_{\delta _{j}}^{d\phi }\tau (\phi
)\right\rangle .  \notag
\end{equation}

By means of the biharmonic equation (\ref{biharmonic_Finsler_to_Riemann}),
this becomes:%
\begin{equation}
-\dfrac{1}{2}\Delta \left\Vert \tau (\phi )\right\Vert ^{2}=\left\langle
trace_{g}\tilde{R}(d\phi ^{\tilde{h}},\tau (\phi ))d\phi ^{\tilde{h}},\tau
(\phi )\right\rangle +g^{ij}\left\langle D_{\delta _{i}}^{d\phi }\tau (\phi
),D_{\delta _{j}}^{d\phi }\tau (\phi )\right\rangle .  \label{Weitz1}
\end{equation}%
According to the hypothesis that the sectional curvature of $(\tilde{M},%
\tilde{g})$ is nonpositive, the curvature term above is nonnegative; since $%
g^{ij}\left\langle D_{\delta _{i}}^{d\phi }\tau (\phi ),D_{\delta
_{j}}^{d\phi }\tau (\phi )\right\rangle $ (as a squared norm) is
nonnegative, too, we get: $-\dfrac{1}{2}\Delta \left\Vert \tau (\phi
)\right\Vert ^{2}\geq 0.$

On the other side, we have, \cite{Zhong}, $\underset{BM}{\int }\Delta
\left\Vert \tau (\phi )\right\Vert ^{2}d\mathcal{V}_{g}=0$, hence, \ $\Delta
\left\Vert \tau (\phi )\right\Vert ^{2}=0;$ thus, by (\ref{Weitz1}),\ $%
g^{ij}\left\langle D_{\delta _{i}}^{d\phi }\tau (\phi ),D_{\delta
_{j}}^{d\phi }\tau (\phi )\right\rangle =0;$ as a consequence, 
\begin{equation}
D_{\delta _{j}}^{d\phi }\tau (\phi )=0.  \label{*}
\end{equation}%
Take the horizontal vector field $X:=(g^{ij}\left\langle \phi _{,i},\tau
(\phi )\right\rangle )\delta _{j}$ on $TM;$ by (\ref{*}), we get:%
\begin{equation*}
0=\underset{BM}{\int }divXd\mathcal{V}_{g}=\underset{BM}{\int }\underset{%
\geq 0}{\underbrace{\left\langle \tau (\phi ),\tau (\phi )\right\rangle }}d%
\mathcal{V}_{g}
\end{equation*}%
and therefore, $\left\langle \tau (\phi ),\tau (\phi )\right\rangle
=0\Rightarrow $ $\tau (\phi )=0,$ i.e., $\phi $ is harmonic.
\end{proof}

Dropping any condition upon the compactness or on the boundary of $M,$ we
have:

\begin{theorem}
Let $(M,g)$ be an arbitrary Finsler space (not necessarily compact), $(%
\tilde{M},\tilde{g}),$ a Riemannian manifold with strictly negative
sectional curvature and $\phi :M\rightarrow \tilde{M},$ a biharmonic map. If 
$\phi $ has the properties: 1)\ $\left\Vert \tau (\phi )\right\Vert =const.$
and 2)\ there exists a point $x_{0}\in M$ at which the rank of $\phi $ is at
least 2, then $\phi $ is harmonic.
\end{theorem}

\begin{proof}
The proof is similar to the one in the Riemannian case, \cite{Oniciuc}. From
the hypothesis $\left\Vert \tau (\phi )\right\Vert =const.,$ in (\ref{Weitz1}%
), the left hand side is 0; but both terms in the right hand side are
nonnegative, hence: $\left\langle trace_{g}\tilde{R}(d\phi ^{\tilde{h}},\tau
(\phi ))d\phi ^{\tilde{h}},\tau (\phi )\right\rangle =0.$ Since $Riem_{%
\tilde{g}}<0,$ we must have, for all $i=\overline{1,n}:$ $d\phi ^{\tilde{h}%
}(\delta _{i})~||~\tau (\phi )$. Taking into account that at $x_{0},$ $%
rank(\phi )\geq 2,$ the only possibility is $\tau (\phi )(x_{0})=0$. Using $%
\left\Vert \tau (\phi )\right\Vert =const,$ it follows that $\tau (\phi
)\equiv 0,$ i.e., $\phi $ is harmonic.
\end{proof}

\section{Biharmonicity of the identity map}

Throughout this section, we assume that $M=\tilde{M}$ (not necessarily
compact), $\dim M=n,$ and denote the coordinates on $TM$ by $(x^{i},y^{i}).$
Considering on $M$ two metrics: a Riemannian one $\tilde{g}$ and a
Finslerian one $g,$ we will explore the biharmonicity of the
Finsler-to-Riemann mapping:%
\begin{equation}
id:(M,g)\rightarrow (M,\tilde{g}).  \label{identity}
\end{equation}

In this situation, there appear two adapted bases $(\delta _{i},\dot{\partial%
}_{i})$ and $(\tilde{\delta}_{i},\dot{\partial}_{i})$ on $TM,$ together with
the covariant differentiations given by $D,$ $\tilde{D}$ and $D^{d(id)}$.
According to \cite{Mo-book}, the tension of the identity map has the local
components 
\begin{equation}
\tau ^{i}(id)=g^{jk}(\tilde{\Gamma}_{~jk}^{i}-G_{~jk}^{i})
\label{tension_id_Mo}
\end{equation}%
(note: our $G^{i}$ is half the one in \cite{Mo-book}).

Let us denote $b:=F^{2}-\tilde{F}^{2},$ i.e.:%
\begin{equation}
g_{ij}(x,y)=\tilde{g}_{ij}(x)+b_{ij}(x,y),  \label{perturbed metric}
\end{equation}%
where the function $b=b(x,y)$ is homogeneous of degree 2 in $y$ and $%
b_{ij}=b_{\cdot ij}.$

In the geodesic equations (\ref{spray_coeff}), we express the derivatives $%
F_{~,k}^{2}$ in terms of $\tilde{D}_{\tilde{\delta}_{i}}$-covariant
derivatives, denoted in the following by double bars $_{||i}~$; we obtain: 
\begin{equation}
2G^{i}=2\tilde{G}^{i}+2B^{i},  \label{geodesic_eqns_perturbation}
\end{equation}%
where:%
\begin{equation}
2B^{i}:=\dfrac{1}{2}g^{ih}(2y_{h||j}y^{j}-F_{~||h}^{2})  \label{expr_B}
\end{equation}%
and $y_{h}:=\dfrac{1}{2}F_{~\cdot h}^{2}=g_{hj}y^{j}.$ The tension of $id$
is:%
\begin{equation}
\tau ^{i}(id)=-g^{jk}B_{\cdot j\cdot k}^{i}.  \label{tau_id}
\end{equation}

A direct computation shows that, in (\ref{expr_B}), the covariant derivative 
$2y_{h||j}$ can be rewritten as:%
\begin{equation}
2y_{h||j}=F_{~||j\cdot h}^{2}.  \label{commutation_F}
\end{equation}

\textbf{Remarks: 1) }If $b$ is parallel with respect to $\tilde{D},$ then $%
F_{~||k}^{2}=\tilde{F}_{~||k}^{2}+b_{||k}=0;$ taking into account (\ref%
{commutation_F}), we get $B^{i}=0\Rightarrow \tau ^{i}=0;$ in this case, the
identity map is harmonic, i.e., also biharmonic.

\textbf{2) }Assuming that $g$ is a Berwald-type metric, i.e., $%
G_{~jk}^{i}=G_{~jk}^{i}(x),$ then there exists, \cite{Matveev}, \cite{Szabo}%
, a Riemannian metric such that $G_{~jk}^{i}=\tilde{G}_{~jk}^{i}$; thus, the
identity map is, again, harmonic, hence, biharmonic.

We will find in the following two examples of Finslerian perturbations $b$
for which the identity of $M$ is proper biharmonic.

With $\tau ^{i}:=\tau ^{i}(id),$ the relation between the $D^{d(id)}$- and $%
\tilde{D}$-covariant derivatives of $\tau ^{i}$ is:%
\begin{equation}
D_{~\delta _{j}}^{d(id)}\tau ^{i}=\delta _{j}\tau ^{i}+\tilde{\Gamma}%
_{~jk}^{i}\tau ^{k}=\tau _{~||j}^{i}-B_{~\cdot j}^{k}\tau _{~\cdot k}^{i}.
\label{rel_derivs}
\end{equation}

\textbf{1. }Suppose that the Finslerian function satisfies:%
\begin{equation}
F_{~||h}^{2}=~\left\langle a,y\right\rangle _{g}y_{h},
\label{perturbation_x}
\end{equation}%
where $\left\langle a,y\right\rangle _{g}:=g_{ij}a^{i}y^{j}$ and $%
a^{i}=a^{i}(x)$ are components of a vector field $A=a^{i}\partial _{i}$ on $M
$ (relations (\ref{perturbation_x}) are equivalent to a first order ODE
system in $g_{ij}$).

A brief calculation using (\ref{commutation_F}) leads to: $%
2y_{j||h}y^{h}-F_{~||j}^{2}=a_{j}F^{2}$, that is, $2B^{i}=a^{i}(x)F^{2}.$
From (\ref{tau_id}), we obtain:%
\begin{equation}
\tau ^{i}=-\dfrac{1}{2}na^{i}.  \label{tau_x}
\end{equation}

Since $\tau ^{i}=\tau ^{i}(x)$, relation (\ref{rel_derivs})\ becomes simply: 
$D_{~\delta _{j}}^{d(id)}\tau ^{i}=\tau _{~||j}^{i}$ and the biharmonic
equation is written as:%
\begin{equation}
g^{jk}(\tau _{~||j||k}^{i}-\tilde{\Gamma}_{~jk}^{l}\tau _{~||l}^{i}-\tilde{R}%
_{j~lk}^{~i}\tau ^{l})=0.  \label{biharmonic_x}
\end{equation}%
Here, taking into account that $\tilde{R}_{jilk}=\tilde{R}_{lkij}$ and Ricci
identities (\ref{Ricci_R}) for $\tilde{D}$, the curvature term $\tilde{R}%
_{j~lk}^{~i}\tau ^{l}$ can be expressed by commuting $\tilde{D}$-covariant
derivatives of $\tau ^{i}.$ It turns out that a sufficient condition for the
biharmonicity of $id$ is:%
\begin{equation}
\tau _{~||j}^{i}=0.  \label{suff_cond_x}
\end{equation}

(\textit{Note:}\ this statement is always true in the Riemannian case, but
generally, not in the Finsler-to-Riemann one, where, as a rule, $\tau
^{i}=\tau ^{i}(x,y)$).

Using (\ref{tau_x}), we deduce that (\ref{suff_cond_x}) is identically
satisfied if the vector field $A^{h}=a^{i}\delta _{i}$ is parallel with
respect to $\tilde{D}.$ But, according to (\ref{Levi-Civita_lift}), this is
nothing but: $\tilde{\nabla}_{\partial _{i}}A=0.$ In other words:

\begin{proposition}
If, in (\ref{perturbation_x}), the nonzero vector field $A=a^{i}(x)\partial
_{i}$ is parallel with respect to $\tilde{g},$ then the identity map $id:(M,%
\tilde{g})\rightarrow (M,g)$ is proper biharmonic.
\end{proposition}

\bigskip 

\textbf{2. Linearized Finslerian perturbations of the Euclidean metric. }%
Assume that $(M,\tilde{g}_{ij})=(\mathbb{R}^{n},\delta _{ij})$ and the
perturbation $b_{ij}=:\varepsilon _{ij}(x,y)$ is small (linearly
approximable), that is, we may neglect all terms of degree greater than one
in $\varepsilon _{ij}$ and its derivatives, \cite{Landau}. In this case, the
inverse metric is given by: $g^{ik}=\delta ^{ik}-\varepsilon ^{ik}$ and
relation (\ref{expr_B}) becomes:%
\begin{equation*}
2B^{i}=\dfrac{1}{2}\delta ^{ih}(\varepsilon _{hj,k}+\varepsilon
_{hk,j}-\varepsilon _{jk,h})y^{j}y^{k}.
\end{equation*}

We notice that the tension $\tau $ will be of the same order of smallness as 
$\varepsilon ;$ it means that products of $\tau $ with $\varepsilon $ and
its derivatives can be neglected. For instance, we have: $B_{~l}^{h}\tau
_{~\cdot h}^{i}\simeq 0,$ which, substituted into (\ref{rel_derivs}), leads
to:%
\begin{equation*}
D_{\delta _{j}}^{d(id)}\tau ^{i}=\tau _{~,j}^{i}.
\end{equation*}%
The biharmonic equation takes the simple form: $\delta ^{lm}\tau
_{~,l,m}^{i}=0.$ Again, a sufficient condition for biharmonicity is%
\begin{equation*}
\tau _{~,l}^{i}=0,
\end{equation*}%
(or: $\tau ^{i}=\tau ^{i}(y)$), that is, $\delta ^{ih}(\varepsilon
_{hj,k,l}+\varepsilon _{hk,j,l}-\varepsilon _{jk,h,l})y^{j}y^{k}=0.$ We
obtain:

\begin{proposition}
Let the Finsler metric $g_{ij}(x,y)=\delta _{ij}+\varepsilon _{ij}(x,y)$ be
a linearized perturbation of the Euclidean metric on $\mathbb{R}^{n}.$ If
the components $\varepsilon _{ij}(x,y)$ are non-constant and affine in $x$,
then $id:(\mathbb{R}^{n},\delta _{ij})\rightarrow (\mathbb{R}^{n},g)$ is
proper biharmonic.
\end{proposition}

\section{Second variation of the bienergy}

Take a biharmonic map $\phi :(M,g)\rightarrow (\tilde{M},\tilde{g})$ and a
smooth 2-parameter variation $f=f(\varepsilon _{1},\varepsilon _{2},x),$ $%
f(0,0,x)=\phi $ of $\phi ,$ with%
\begin{equation*}
\mathbf{V}_{1}=df^{\tilde{h}}(\partial _{\varepsilon _{1}}),~~\mathbf{V}%
_{2}=df^{\tilde{h}}(\partial _{\varepsilon _{2}}),~\ \ V_{1}:=\mathbf{V}%
_{1|\varepsilon _{1}=\varepsilon _{2}=0},~\ V_{2}:=\mathbf{V}_{1|\varepsilon
_{1}=\varepsilon _{2}=0}
\end{equation*}%
(if $M$ has a boundary, then $V_{1},$ $V_{2}$ and their $\delta _{i}$%
-covariant derivatives are assumed to vanish on $\partial M$).

The deduction of the second variation of $E_{2}$ follows the same steps as
in the Riemannian case, with two differences: in the expressions of $\tau
(f) $ and of $\Delta ^{df},$ there appear extra terms and we have to take
into account that $\Phi _{\ast }(\partial _{\varepsilon _{i}})$ is,
generally, not horizontal. Fortunately, as we will see below, these will
finally not complicate the expression of the variation.

We denote, for simplicity, $\tau :=\tau (f).$ According to (\ref%
{first_var_bienergy}), (\ref{biharmonic_Finsler_to_Riemann}): 
\begin{equation}
\dfrac{\partial E_{2}(f)}{\partial \varepsilon _{1}}=\underset{BM}{\int }%
\left\langle \tau _{2}(f),\mathbf{V}_{1}\right\rangle d\mathcal{V}_{g};
\end{equation}%
differentiating with respect to $\varepsilon _{2}:$%
\begin{equation*}
\dfrac{\partial ^{2}E_{2}(f)}{\partial \varepsilon _{1}\partial \varepsilon
_{2}}=\underset{BM}{\int }\{\left\langle D_{\partial _{\varepsilon
_{2}}}^{df}\tau _{2}(f),\mathbf{V}_{1}\right\rangle +\left\langle \tau
_{2}(f),D_{\partial _{\varepsilon _{2}}}^{df}\mathbf{V}_{1}\right\rangle \}d%
\mathcal{V}_{g}.
\end{equation*}%
At $\varepsilon _{1}=\varepsilon _{2}=0,$ since $\phi $ is biharmonic, the
second term in the right hand side term will vanish. It is thus enough to
evaluate the first one; we have: 
\begin{equation}
D_{\partial _{\varepsilon _{2}}}^{df}\tau _{2}(f)=-D_{\partial _{\varepsilon
_{2}}}^{df}(\Delta ^{df}\tau )-D_{\partial _{\varepsilon
_{2}}}^{df}(trace_{g}\tilde{R}(df^{\tilde{h}},\tau )df^{\tilde{h}}).
\label{to_evaluate}
\end{equation}

The covariant derivative of the Laplacian $-\Delta ^{df}\tau $ is:%
\begin{equation}
T_{1}:=-D_{\partial _{\varepsilon _{2}}}^{df}(\Delta ^{df}\tau
)=g^{ij}D_{\partial _{\varepsilon _{2}}}^{df}\left( D_{~\delta
_{i}}^{df}D_{~\delta _{j}}^{df}\tau -D_{~D_{\delta _{i}}\delta
_{j}}^{df}\tau -P_{i}D_{\delta _{j}}^{df}\tau )\right) .
\label{deriv_Laplacian}
\end{equation}

Commuting covariant derivatives by means of the curvature $\tilde{R}$ (twice
for the term $D_{~\delta _{i}}^{df}D_{~\delta _{j}}^{df}\tau $), taking into
account that $[\partial _{\varepsilon },\delta _{i}]=0$ and (\ref{Riem_R}),
we find: 
\begin{equation*}
\begin{array}{c}
T_{1}=g^{ij}{\Large \{}\tilde{R}(\mathbf{V}_{2},df^{\tilde{h}}(\delta
_{i}))D_{~\delta _{j}}^{df}\tau +D_{~\delta _{i}}^{df}\left( \tilde{R}(%
\mathbf{V}_{2},df^{\tilde{h}}(\delta _{j}))\tau +D_{~\delta
_{j}}^{df}D_{\partial _{\varepsilon _{2}}}^{df}\tau \right) - \\ 
-\tilde{R}(\mathbf{V}_{2},df^{\tilde{h}}(D_{\delta _{i}}\delta _{j}))\tau
-D_{~D_{\delta _{i}}\delta _{j}}^{df}D_{\partial _{\varepsilon
_{2}}}^{df}\tau -P_{i}\tilde{R}(\mathbf{V}_{2},df^{\tilde{h}}(\delta
_{j}))\tau ~-P_{i}D_{\delta _{j}}^{df}D_{\partial _{\varepsilon
_{2}}}^{df}\tau {\Large \}}.%
\end{array}%
\end{equation*}

The terms in $D_{\partial _{\varepsilon _{2}}}^{df}\tau $ can be grouped
into $-\Delta ^{df}(D_{\partial _{\varepsilon _{2}}}^{df}\tau ):$ 
\begin{equation*}
\begin{array}{c}
T_{1}=-\Delta ^{df}(D_{\partial _{\varepsilon _{2}}}^{df}\tau )+g^{ij}%
{\Large \{}\tilde{R}(\mathbf{V}_{2},df^{\tilde{h}}(\delta _{i}))D_{~\delta
_{j}}^{df}\tau + \\ 
D_{~\delta _{i}}^{df}\left( \tilde{R}(\mathbf{V}_{2},df^{\tilde{h}}(\delta
_{j}))\tau \right) -\tilde{R}(\mathbf{V}_{2},df^{\tilde{h}}(D_{\delta
_{i}}\delta _{j}))\tau -P_{i}\tilde{R}(\mathbf{V}_{2},df^{\tilde{h}}(\delta
_{j}))\tau {\Large \}}.%
\end{array}%
\end{equation*}

Splitting $D_{~\delta _{i}}^{df}\left( \tilde{R}(\mathbf{V}_{2},df^{\tilde{h}%
}(\delta _{j}))\tau \right) $ as a sum of derivatives, we recognize in the
resulting expression $\tilde{R}(\mathbf{V}_{2},\tau )\tau :$ 
\begin{equation}
\begin{array}{l}
T_{1}=-\Delta ^{df}(D_{\partial _{\varepsilon _{2}}}^{df}\tau )+\tilde{R}(%
\mathbf{V}_{2},\tau )\tau +g^{ij}{\Large \{}(D_{~\delta _{i}}^{df}\tilde{R})(%
\mathbf{V}_{2},df^{\tilde{h}}(\delta _{j}))\tau + \\ 
+2\tilde{R}(\mathbf{V}_{2},df^{\tilde{h}}(\delta _{i}))D_{~\delta
_{j}}^{df}\tau +\tilde{R}(D_{~\delta _{i}}^{df}\mathbf{V}_{2},df^{\tilde{h}%
}(\delta _{j}))\tau {\Large \}}.%
\end{array}
\label{T1}
\end{equation}

The curvature term $T_{2}:=-D_{\partial _{\varepsilon _{2}}}^{df}(trace_{g}%
\tilde{R}(df^{\tilde{h}},\tau )df^{\tilde{h}})$ in (\ref{to_evaluate}) is:%
\begin{equation}
\begin{array}{c}
T_{2}=-g^{ij}\{(D_{\partial _{\varepsilon _{2}}}^{df}\tilde{R})(df^{\tilde{h}%
}(\delta _{i}),\tau )df^{\tilde{h}}(\delta _{j})+\tilde{R}(D_{\partial
_{\varepsilon _{2}}}^{df}df^{\tilde{h}}(\delta _{i}),\tau )df^{\tilde{h}%
}(\delta _{j})+ \\ 
+\tilde{R}(df^{\tilde{h}}(\delta _{i}),D_{\partial _{\varepsilon
_{2}}}^{df}\tau )df^{\tilde{h}}(\delta _{j})+\tilde{R}(df^{\tilde{h}}(\delta
_{i}),\tau )D_{\partial _{\varepsilon _{2}}}^{df}(df^{\tilde{h}}(\delta
_{j})).%
\end{array}%
\end{equation}

Taking into account that $\tilde{R}=\tilde{R}(x)$ only, we obtain $%
D_{\partial _{\varepsilon _{2}}}^{df}\tilde{R}=\tilde{D}_{\mathbf{V}_{2}}%
\tilde{R}.$ Transforming $D_{\partial _{\varepsilon _{2}}}^{df}df^{\tilde{h}%
}(\delta _{i}),$ $D_{\partial _{\varepsilon _{2}}}^{df}df^{\tilde{h}}(\delta
_{j})$ by (\ref{commutation_rule_2}) and then using first Bianchi identity
in the second term:%
\begin{equation}
\begin{array}{c}
T_{2}=-g^{ij}\{(\tilde{D}_{\mathbf{V}_{2}}\tilde{R})(df^{\tilde{h}}(\delta
_{i}),\tau )df^{\tilde{h}}(\delta _{j})+\tilde{R}(D_{\delta _{i}}^{df}%
\mathbf{V}_{2},\tau )df^{\tilde{h}}(\delta _{j})+ \\ 
+\tilde{R}(df^{\tilde{h}}(\delta _{i}),D_{\partial _{\varepsilon
_{2}}}^{df}\tau )df^{\tilde{h}}(\delta _{j})+\tilde{R}(df^{\tilde{h}}(\delta
_{i}),\tau )\tilde{D}_{\delta _{j}}\mathbf{V}_{2}\}= \\ 
=-g^{ij}\{(\tilde{D}_{\mathbf{V}_{2}}\tilde{R})(df^{\tilde{h}}(\delta
_{i}),\tau )df^{\tilde{h}}(\delta _{j})+2\tilde{R}(df^{\tilde{h}}(\delta
_{i}),\tau )\tilde{D}_{\delta _{j}}\mathbf{V}_{2}- \\ 
-\tilde{R}(df^{\tilde{h}}(\delta _{j}),D_{\delta _{i}}^{df}\mathbf{V}%
_{2})\tau )+\tilde{R}(df^{\tilde{h}}(\delta _{i}),D_{\partial _{\varepsilon
_{2}}}^{df}\tau )df^{\tilde{h}}(\delta _{j})\}.%
\end{array}%
\end{equation}

Second, and then first Bianchi identities for the $(\tilde{D}_{\mathbf{V}%
_{2}}\tilde{R})$-term tell us that:%
\begin{equation*}
\begin{array}{c}
-g^{ij}(\tilde{D}_{\mathbf{V}_{2}}\tilde{R})(df^{\tilde{h}}(\delta
_{i}),\tau )df^{\tilde{h}}(\delta _{j})=g^{ij}\{(\tilde{D}_{\tau }\tilde{R})(%
\mathbf{V}_{2},df^{\tilde{h}}(\delta _{i}))df^{\tilde{h}}(\delta _{j})- \\ 
-(\tilde{D}_{\delta _{i}}\tilde{R})(df^{\tilde{h}}(\delta _{j}),\tau )%
\mathbf{V}_{2}-(\tilde{D}_{\delta _{i}}\tilde{R})(\mathbf{V}_{2},df^{\tilde{h%
}}(\delta _{j}))\tau \}.%
\end{array}%
\end{equation*}

Substituting into $T_{2}$ and adding: $T_{1}+T_{2}=D_{\partial _{\varepsilon
_{2}}}^{df}\tau _{2}(f),$ we get:%
\begin{equation*}
\begin{array}{l}
D_{\partial _{\varepsilon _{2}}}^{df}\tau _{2}(f)=\mathcal{J}(D_{\partial
_{\varepsilon _{2}}}^{df}\tau )+\tilde{R}(\mathbf{V}_{2},\tau )\tau +g^{ij}%
{\Large \{}(\tilde{D}_{\tau }\tilde{R})(\mathbf{V}_{2},df^{\tilde{h}}(\delta
_{i}))df^{\tilde{h}}(\delta _{j})- \\ 
-(D_{~\delta _{i}}^{df}\tilde{R})(df^{\tilde{h}}(\delta _{j}),\tau )\mathbf{V%
}_{2}+2\tilde{R}(\mathbf{V}_{2},df^{\tilde{h}}(\delta _{i}))D_{~\delta
_{j}}^{df}\tau -2\tilde{R}(df^{\tilde{h}}(\delta _{i}),\tau )D_{~\delta
_{i}}^{df}\mathbf{V}_{2}{\Large \}},%
\end{array}%
\end{equation*}%
with $\mathcal{J}$ as in (\ref{rough_Laplacian}). Using (\ref{deriv_tau_1})
and evaluating at $\varepsilon _{1}=\varepsilon _{2}=0,$ we get:

\begin{proposition}
The second variation of the bienergy of a Finsler-to-Riemann biharmonic map $%
\phi :M\rightarrow \tilde{M}$ is:%
\begin{eqnarray}
&&\dfrac{\partial ^{2}E_{2}(f)}{\partial \varepsilon _{1}\partial
\varepsilon _{2}}|_{\varepsilon _{1}=\varepsilon _{2}=0}=\underset{BM}{\int }%
\left\langle V_{1},\right. \mathcal{J}^{2}V_{2}+\tilde{R}(V_{2},\tau )\tau +
\label{second variation_E2} \\
&&+g^{ij}{\Large \{}(\tilde{D}_{\tau }\tilde{R})(\mathbf{V}_{2},df^{\tilde{h}%
}(\delta _{i}))df^{\tilde{h}}(\delta _{j})-(D_{~\delta _{i}}^{df}\tilde{R}%
)(df^{\tilde{h}}(\delta _{j}),\tau )\mathbf{V}_{2}+  \notag \\
&&+2\tilde{R}(\mathbf{V}_{2},df^{\tilde{h}}(\delta _{i}))D_{~\delta
_{j}}^{df}\tau -2\tilde{R}(df^{\tilde{h}}(\delta _{i}),\tau )D_{~\delta
_{i}}^{df}\mathbf{V}_{2}{\Large \}}\left. {}\right\rangle d\mathcal{V}_{g}. 
\notag
\end{eqnarray}
\end{proposition}

In particular cases (for instance, $\tilde{M}=\mathbb{R}^{n}$ or $S^{n}$), (%
\ref{second variation_E2})\ becomes considerably simpler.

The Hessian $\mathcal{H}:(V_{1},V_{2})\mapsto \mathcal{H}(V_{1},V_{2})=%
\dfrac{\partial ^{2}E_{2}(f)}{\partial \varepsilon _{1}\partial \varepsilon
_{2}}|_{\varepsilon _{1}=\varepsilon _{2}=0}$ of the bienergy is a symmetric
bilinear form. A solution $\phi $ of the biharmonic equation is \textit{%
stable} if the quadratic form $\mathcal{H}(V,V)$ is nonnegative for any $V$.
As an example, harmonic maps are stable biharmonic maps.

\textbf{Acknowledgments. }1) The work was supported by the Sectorial
Operational Program Human Resources Development (SOP HRD), financed from the
European Social Fund and by Romanian Government under the Project number
POSDRU/89/1.5/S/59323.

2)\ Special thanks to prof. G. Munteanu for proofreading the text.

\end{document}